\documentclass[a4paper,onecolumn,10pt]{amsart}

\usepackage{amsmath}
\usepackage{amssymb}  
\usepackage{amsthm}
\usepackage{amsfonts}
\usepackage{graphicx}
\usepackage{cancel}
\usepackage{bbm}
\usepackage{url}
\usepackage{enumerate} 
\usepackage{ upgreek }
\usepackage{dsfont}
\usepackage{color}
\usepackage{subcaption}
\usepackage{comment}

\usepackage{makecell}
\usepackage{diagbox}
\usepackage{wrapfig}
\usepackage{rotating}
\usepackage{tabularx}

\usepackage{hyperref}
\hypersetup{
    colorlinks=true,
    linkcolor=blue,
    citecolor=red,
    filecolor=magenta,      
    urlcolor=cyan,
}

\newcommand{\braket}[2]{\langle #1 , #2 \rangle}

\newcommand{\eps}{\varepsilon}

\def\beq{\begin{equation}}
\def\eeq{\end{equation}}
\def\bq{\begin{quote}}
\def\eq{\end{quote}}
\def\ben{\begin{enumerate}}
\def\een{\end{enumerate}}
\def\bit{\begin{itemize}}
\def\eit{\end{itemize}}

\def\ra{\rightarrow}

\def\lb{\left(}
\def\rb{\right)}
\def\lset{\lbrace}
\def\rset{\rbrace}

\def\r|{\right|}
\def\lbr{\left[}
\def\rbr{\right]}
\def\ident{\textnormal{id}}

\newcommand\R{\mathbbm{R}}

\newcommand\N{\mathbbm{N}}

\newcommand{\scalar}[2]{\langle#1,#2\rangle}

\DeclareMathOperator{\Tr}{Tr}

\DeclareMathOperator{\mdim}{dim}

\theoremstyle{plain}
\newtheorem{thm}{Theorem}[section]
\newtheorem{lem}[thm]{Lemma}
\newtheorem{cor}[thm]{Corollary}
\newtheorem{prop}[thm]{Proposition}
\newtheorem{defn}[thm]{Definition}

\theoremstyle{definition}
\newtheorem{example}{Example}
\newtheorem{rem}{Remark}[section]

\newcommand{\iy}{\infty}
\newcommand{\e}{\varepsilon}
\renewcommand{\leq}{\leqslant}
\renewcommand{\geq}{\geqslant}

\newcommand\name{tensor radius }
\newcommand\names{tensor radii }

\newcommand\Names{Tensor radii }

\begin{document}

\title{{Asymptotic Tensor Powers of Banach Spaces}}

\author{Guillaume Aubrun}
\address{\small{Institut Camille Jordan, Universit\'{e} Claude Bernard Lyon 1, 43 boulevard du 11 novembre 1918, 69622 Villeurbanne cedex, France}}
\email{aubrun@math.univ-lyon1.fr}
\author{Alexander M\"uller-Hermes}
\address{\small{Institut Camille Jordan, Universit\'{e} Claude Bernard Lyon 1, 43 boulevard du 11 novembre 1918, 69622 Villeurbanne cedex, France}}
\address{\small{Department of Mathematics, University of Oslo, P.O. box 1053, Blindern, 0316 Oslo, Norway}}
\email{muellerh@math.uio.no, muellerh@posteo.net}

\maketitle
\date{\today}
\begin{abstract}
We study the asymptotic behaviour of large tensor powers of normed spaces and of operators between them. We define the \name of a finite-dimensional normed space $X$ as the limit of the sequence $A_k^{1/k}$, where $A_k$ is the equivalence constant between the projective and injective norms on $X^{\otimes k}$. We show that Euclidean spaces are characterized by the property that their \name equals their dimension. Moreover, we compute the \name for spaces with enough symmetries, such as the spaces $\ell_p^n$. We also define the \name of an operator $T$ as the limit of the sequence $B_k^{1/k}$, where $B_k$ is the injective-to-projective norm of $T^{\otimes k}$. We show that the \name of an operator whose domain or range is Euclidean is equal to its nuclear norm, and give some evidence that this property might characterize Euclidean spaces.
\end{abstract}

\section{Introduction}
For a real or complex Banach space $X$ we denote its dual space by $X^*$ and their respective unit balls by $B_X$ and $B_{X^*}$. Two natural norms can be defined on the algebraic tensor product $X^{\otimes k}$: The \emph{injective tensor norm} is given by
\[
\|z\|_{\e_k(X)} = \sup\Big\lset |(\lambda_1\otimes \cdots \otimes \lambda_k)(z)| ~:~\lambda_1,\ldots ,\lambda_k\in B_{X^*}\Big\rset,
\]
for $z\in X^{\otimes k}$, and the \emph{projective tensor norm} by
\[
\|z\|_{\pi_k(X)} = \inf\Big\lset \sum^n_{i=1} \|x^{(1)}_i\|_X\cdots \|x^{(k)}_i\|_X ~:~n\in\N,\ z=\sum^n_{i=1} x^{(1)}_{i}\otimes\ldots \otimes x^{(k)}_{i}\Big\rset .
\]
When $X$ is finite-dimensional, the discrepancy between these norms is governed by the parameter
\[ \rho_k(X) = \lb \sup_{z\in X^{\otimes k}}\frac{\|z\|_{\pi_k(X)}}{\|z\|_{\e_k(X)}}\rb^{\frac{1}{k}}.
\]
More generally, given a linear operator $T: X \to Y$ between finite-dimensional normed spaces, we consider
\begin{equation} \label{eq:def-tau} \tau_k(T) = \| T^{\otimes k} \|^{1/k}_{\e_k(X) \to \pi_k(Y)} = \lb \sup_{z\in X^{\otimes k}}\frac{\|T^{\otimes k}z\|_{\pi_k(Y)}}{\|z\|_{\e_k(X)}}\rb^{\frac{1}{k}},
\end{equation}
such that $\rho_k(X)=\tau_k(\mathrm{id}_X)$. A standard subadditivity argument (see Lemma \ref{lem:subadditivity} below) shows the existence of the limits
\[ \rho_{\infty}(X)= \lim_{k \to \infty} \rho_k(X), \ \ 
\tau_{\infty}(T) = \lim_{k \to \infty} \tau_k(T). \]
We call these limits the \emph{\name of the space $X$} and the \emph{\name of the operator $T$}, respectively, by analogy with the spectral radius formula. The \names are motivated by questions in quantum information theory where they have recently been applied by the authors~\cite{paperA}. In this article we aim to understand the properties of \names and we begin by stating our main results.

\subsection*{Main results on \names of normed spaces:} We show that the \name is maximal precisely for Euclidean spaces:
\begin{thm}
If $X$ is an $n$-dimensional normed space, then $\sqrt{n} \leq \rho_\infty(X) \leq n$. Moreover, $\rho_{\infty}(X)=n$ if and only if $X$ is Euclidean.
\label{thm:Main1}
\end{thm}

We will state the proof of Theorem \ref{thm:Main1} in Section \ref{sec:proof-th1}. For a large class of normed spaces, including the spaces $\ell_p^n$ and their noncommutative analogues, we compute the \name as a function of the Banach--Mazur distance (see Section~ \ref{subsection:Bm-distance}) to the Euclidean space. Our argument applies to normed spaces with proportional John and Loewner ellipsoids (see Section~\ref{subsection:john}), which includes all normed spaces with enough symmetries (see Section \ref{subsec:enough-symmetries}). 

\begin{thm} \label{thm:enough-symmetries}
If $X$ is an $n$-dimensional normed space with proportional John and Loewner ellipsoids, then 
\[\rho_{\infty}(X) = \frac{n}{\mathrm{d}(X,\ell_2^n)} .\]
\end{thm}

The proof of Theorem \ref{thm:enough-symmetries} is given in Section \ref{sec:proof-th2}. In particular, we see that $\rho_{\infty}(\ell_{\infty}^n)=\sqrt{n}$, showing that the lower bound in Theorem \ref{thm:Main1} is sharp.

\subsection*{Main results on \names of linear operators:} Let $T : X \to Y$ be a linear operator between finite-dimensional normed spaces and let $\|T\|_{N(X \to Y)}$ denote its nuclear norm. It is elementary to show (see Section \ref{sec:tau}) that
\begin{equation} \label{eq:tau-inequalities}\|T\|_{X \to Y} = \tau_1(T)  \leq \tau_{\infty}(T) \leq \|T\|_{N(X \to Y)}.\end{equation}
In order to understand when the upper bound in \eqref{eq:tau-inequalities} is sharp, we introduce the following property:
\begin{defn}\label{defn:NTP}
A pair of finite-dimensional normed spaces $(X,Y)$ is said to have the \emph{nuclear tensorization property} (NTP) if $\tau_\infty(T) = \|T\|_{N(X \to Y)}$ holds for every operator $T : X \to Y$.
\end{defn}
 
 It is an elementary fact (see \eqref{eq:nuclear-norm-id} below) that $\|\ident_X\|_{N(X\ra X)}=n$ for the identity operator $\ident_X:X\ra X$ on any $n$-dimensional normed space $X$. Since $\tau_\infty(\ident_X)=\rho_\infty(X)$, Theorem \ref{thm:Main1} implies that the pair $(X,X)$ fails to have the NTP whenever $X$ is not Euclidean. The example of $(\ell^2_\infty,\ell^2_\infty)$ is elementary and quite instructive. We will state it here, before developing the general theory later:
 
 \begin{example} \label{example:hadamard}
In the following example, all spaces are over the reals. Using the isometric isomorphism between $\ell_1^2$ and $\ell_\iy^2$, computing $\rho_\infty(\ell_\infty^2)$ is equivalent to computing $\tau_\infty(H)$ of the Hadamard matrix
\[
H = \frac 12 \begin{pmatrix} 1 & 1 \\ 1 & -1 \end{pmatrix},
\]
seen as an operator from $\ell_\infty^2$ to $\ell_1^2$. We have $\tau_1(H)=\|H\|_{\ell_\infty^2\ra \ell_1^2}=1$ and it is elementary to compute $\|H\|_{N(\ell_\infty^2 \to \ell_1^2)}=2$ (e.g., using \cite[Proposition 8.7]{tomczak1989banach}). Consider $k\in\N$ and identify $\e_k(\ell_\infty^2)$ with $\ell_\infty^{2^k}$, as well as their dual spaces. We find
\begin{equation} \label{eq:hadamard} \tau_k(H)^k = \sup_{\alpha,\beta \in \{-1,1\}^{2^k}} \langle \alpha , H^{\otimes k} \beta \rangle \leq \left(\sqrt{2}\right)^k, \end{equation}
by the Cauchy--Schwarz inequality and the fact that $2^{k/2}H^{\otimes k}$ is an orthogonal matrix. Consider the vector $x=(1,1,1,-1)$, which is an eigenvector for $H^{\otimes 2}$ with eigenvalue $1/2$. When $k=2p$ is even, the choice $\alpha=\beta=x^{\otimes p}$ shows that $\tau_k(H)=\sqrt{2}$. When $k$ is odd, the inequality in \eqref{eq:hadamard} is strict (the left-hand side is a rational number and the right-hand side is irrational) and therefore $\tau_k(H) < \sqrt{2}$. We have $\tau_\infty(H) = \sqrt{2}  < 2 = \|H\|_{N(\ell_\infty^2 \to \ell_1^2)}$ and $(\ell^2_\infty , \ell^2_1)$ does not have the NTP.
\end{example}
 
In Section \ref{sec:SpacesWithoutNTP} we find that many natural examples of pairs of normed spaces do not have the NPT. However, there are pairs of distinct normed spaces $(X,Y)$ which have the NTP. For example, this is the case when either $X$ or $Y$ is Euclidean:

\begin{thm} \label{thm:main-nuclear-norm}
If $X$, $Y$ are finite-dimensional normed spaces and one of them is Euclidean, then $(X,Y)$ has the NTP.
\end{thm}

We prove Theorem \ref{thm:main-nuclear-norm} in Section \ref{sec:tau=nuclear-for-euclidean} and in Section \ref{sec:infinite} we generalize it to infinite-dimensional Banach spaces. It is a natural question, whether there exist pairs of distinct non-Euclidean spaces with the NTP. We leave this question open.

\subsection*{Structure of the paper}

Section \ref{sec:notation} gathers background from Banach spaces theory. Section \ref{sec:tau} discusses elementary properties of the \names. A crucial ingredient to all our arguments is Lemma \ref{lem:UpperBound}, which gives an upper bound on the \name of an operator in terms of its factorization through Euclidean spaces. This lemma is especially powerful when combined with the John and/or Loewner ellipsoids of normed spaces. 

In Section \ref{sec:rho}, we focus on the \name of a space and prove Theorems \ref{thm:Main1} and \ref{thm:enough-symmetries}. The first step is Theorem \ref{thm:euclidean}, which states that the \name of the space $\ell_2^n$ equals $n$. We obtain this result by showing that a uniformly chosen random vector is typically ``very entangled". 

In Section \ref{sec:tau=nuclear-for-euclidean} we study the \names of linear operators and focus on when they coincide with the nuclear norm. In Section \ref{sec:ProofofMain-nuclear-norm} we show that the \name for a pair $(X,Y)$ of normed spaces coincides with the nuclear norm when either $X$ or $Y$ is Euclidean (thereby proving Theorem \ref{thm:main-nuclear-norm}). In Section \ref{sec:SpacesWithoutNTP} we we identify examples of spaces where the \name does not coincide with the nuclear norm. 

In Section \ref{sec:FurtherQuestions} we study some natural questions about the \name including whether it is a continuous or a norm (see Section \ref{sec:Norm}), whether it is multiplicative (see Section \ref{sec:submultiplicativity}), or for which spaces it attains its minimal possible value (see Section \ref{sec:Minimality}). While we answer some of these questions, we leave open many questions for future research. In Section \ref{sec:infinite} we discuss extensions of our work to infinite dimensional Banach spaces.

\section{Notation and preliminaries} \label{sec:notation}

In most of the paper, we restrict to finite-dimensional normed spaces. Extensions to infinite dimensions are briefly discussed in Section \ref{sec:infinite}. Let $X$ be a finite-dimensional real or complex normed space $X$. We denote its unit ball by $B_X$ and its dual space by $X^*$. If $Y$ is another finite-dimensional normed space, we denote by $L(X,Y)$ the space of linear operators from $X$ to $Y$. The operator norm of $T \in L(X,Y)$ is denoted $\|T\|_{X \to Y}$ or simply $\|T\|$.

\subsection{Tensor norms}

Let $X$, $Y$ be finite-dimensional normed spaces. A \emph{crossnorm} is a norm $\|\cdot\|$ on  $X \otimes Y$ satisfying the conditions
\[ \| x \otimes y \| = \|x\|_X \|y\|_Y  \textnormal{ and } \| x^* \otimes y^* \|_* = \|x^*\|_{X^*}  \|y\|_{Y^*}\]
for every $x \in X$, $y \in Y$, $x^* \in X^*$, $y^* \in Y^*$, where $\|\cdot\|_*$ is the dual norm to $\|\cdot\|$. Important examples of crossnorms are the \emph{injective norm} given by
\[
\|z\|_{\e} = \sup\Big\lset |(x^* \otimes y^*)(z)| ~:~x^* \in B_{X^*}, \ y^* \in B_{Y^*} \Big\rset,
\]
for $z\in X \otimes Y$, and the projective norm, defined as
\[
\|z\|_\pi = 
\inf\Big\lset \sum^n_{i=1} \|x_i\|_X \|y_i\|_Y ~:~n\in\N, z=\sum^n_{i=1} x_{i} \otimes y_{i}\Big\rset .
\]
We denote by $X \otimes_\e Y$ (resp., $X \otimes_\pi Y$) the space $X \otimes Y$ equipped with the injective (resp., projective). These norms are in duality: $(X \otimes_\pi Y)^*$ identifies with $X^* \otimes_\e Y^*$ and $(X \otimes_\e Y)^*$ identifies with $X^* \otimes_\pi Y^*$. All these definitions and properties have natural extensions to the tensor product of more than two spaces, leading to the definition of $\|\cdot\|_{\e_k(X)}$ and $\|\cdot\|_{\pi_k(X)}$ given in the introduction. We will denote by $\e_k(X)$ and $\pi_k(X)$ the space $X^{\otimes k}$ equipped with the injective and projective norm, respectively.

In the special case of Euclidean spaces, another crossnorm plays a special role: If $X, Y$ are Euclidean spaces, equipped with inner products $\langle \cdot , \cdot \rangle_X$ and $\langle \cdot , \cdot \rangle_Y$, we may define uniquely an inner product on $X \otimes Y$ by the formula
\[\langle x_1 \otimes y_1, x_2 \otimes y_2 \rangle = \langle x_1,x_2 \rangle_X \langle y_1,y_2 \rangle_Y\]
for $x_1$, $x_2$ in $X$ and $y_1$, $y_2$ in $Y$. The corresponding Euclidean norm, called the Hilbert--Schmidt norm and denoted $\|\cdot\|_{\mathrm{HS}}$, is a crossnorm. This definition extends immediately to $k \geq 2$ factors and we denote by $\|\cdot\|_{\mathrm{HS}_k(X)}$ the Hilbert--Schmidt norm on $X^{\otimes k}$.

A \emph{tensor norm} $\alpha$ is the data, for each pair $(X,Y)$ of finite-dimensional normed spaces, of a crossnorm $\|\cdot\|_{X \otimes_\alpha Y}$ on $X \otimes Y$, satisfying the following axiom called the \emph{metric mapping property}: 
if $X$, $X'$, $Y$, $Y'$ are finite-dimensional normed space and $S \in L(X,X')$, $T \in L(Y,Y')$, then 
\[ \|S \otimes T\|_{X \otimes_\alpha Y \to X' \otimes_\alpha Y' } =  \|S\| \cdot \|T\|.\]
Both the injective norm $\e$ and the projective norm $\pi$ are tensor norms.

\subsection{Operators, nuclear norm and trace duality}

Given finite-dimensional normed spaces $X, Y$, we denote by $L(X,Y)$ the space of linear operators from $X$ to $Y$. The \emph{operator norm} of an operator $T \in L(X,Y)$ is given by
\[ \|T\|_{X \to Y} = \sup_{x \in B_X} \|Tx\|_Y ,\]
and its \emph{nuclear norm} by
\[ \|T\|_{N(X \to Y)} = \inf\, \sum_{i} \|y_i\|_Y\|x^*_i\|_{X^*} ,
\]
where the infimum is over so-called nuclear decompositions $T(\cdot)=\sum_{i} x^*_i(\cdot)y_{i}$ with $y_1,\ldots ,y_n\in Y$ and $x^*_1,\ldots , x^*_n\in X^*$. We may write $\|T\|_N$ instead of $\|T\|_{N(X \to Y)}$ if there is no ambiguity. The operator and nuclear norms satisfy the \emph{ideal property}: given operators $T \in L(X_0,X)$, $S \in L(X,Y)$, $R \in L(Y,Y_0)$, we have
\[ \|RST\|_{X_0 \to Y_0} \leq \|R\|_{Y \to Y_0} \|S\|_{X \to Y} \|T\|_{X_0 \to X}, \]
\[ \|RST\|_{N(X_0 \to Y_0)} \leq \|R\|_{Y \to Y_0} \|S\|_{N(X \to Y)} \|T\|_{X_0 \to X}. \]
The space $L(X,Y)$ can be canonically identified with $X^* \otimes Y$; under this identification, the operator and nuclear norms on $L(X,Y)$ correspond respectively to the injective and projective norms on $X^* \otimes Y$. Then, the duality between injective and projective norms translates into the \emph{trace duality}: for any $T \in L(X,Y)$, we have
\begin{equation} 
\label{eq:trace-duality}
\|T\|_{N(X \to Y)} = \sup_{\|Q\|_{Y \to X} \leq 1} \Tr \lbr Q \circ T \rbr.
\end{equation}

\subsection{John and Loewner ellipsoids} \label{subsection:john}

Here, we review standard facts about John and Loewner ellipsoids and refer to \cite[\S 15]{tomczak1989banach} for details.
Let $X$ be an $n$-dimensional normed space. An ellipsoid in $X$ is a linear image of the Euclidean unit ball $B_2^n \subset \R^n$. Among all ellipsoids contained in $B_X$, there is a unique ellipsoid of maximal volume called the \emph{John ellipsoid} of $X$. Let $\|\cdot\|_J$ be the Euclidean norm induced by the John ellipsoid; it satisfies $\|\cdot\|_X \leq \|\cdot\|_J$. A \emph{John contact point} of $X$ is a vector $x \in X$ such that $\|x\|=\|x\|_J=1$.

Among all ellipsoids containing $B_X$, there is a unique ellipsoid of minimal volume called the \emph{Loewner ellipsoid} of $X$. Let $\|\cdot\|_L$ be the Euclidean norm induced by the John ellipsoid; it satisfies $\|\cdot\|_X \geq \|\cdot\|_L$. A \emph{Loewner contact point} of $X$ is a vector $x \in X$ such that $\|x\|=\|x\|_L=1$.

John's theorem asserts that any normed space contains sufficiently many John contact points to allow for a decomposition of the identity. It is most conveniently formulated for normed spaces $(\R^n,\|\cdot\|)$ whose John ellipsoid equals $B_2^n$. This is not a restriction since any $n$-dimensional normed space is isometric to such a space.

\begin{thm} \label{thm:john-lowner}
Let $\|\cdot\|$ be a norm on $\R^n$ such that $B_2^n$ is the John ellipsoid for $(\R^n,\|\cdot\|)$. Then, there are $c_i > 0$ and John contact points $v_i \in \R^n$ such that
\[ x = \sum_{i} c_i \scalar{v_i}{x} v_i, \]
for every $x \in \R^n$.
\end{thm}

Consider a normed space $X=(\R^n,\|\cdot\|)$. We identify the dual space $X^*$ with $(\R^n,\|\cdot\|_*)$, where
$\|x\|_* = \sup \{ |\scalar{x}{y}| \, ; \, \|y\|\leq 1 \}$. Assume that $B_2^n$ is the John ellipsoid of $X$. Then it is also the Loewner ellipsoid of $X^*$. Moreover, John contact points of $X$ coincide with Loewner contact points of $X^*$. If we decompose $\ident_X$ as $\sum_i c_i \scalar{v_i}{\cdot}v_i$ for John/Loewner contact points $v_i$, we have $\Tr\lbr \ident_X\rbr=\sum_{i}c_i = n$. Using this decomposition and trace-duality, together with the fact that $\|\ident_{\R^n}\|_{X^*\ra X}\leq 1$, shows that
\begin{equation} \label{eq:nuclear-norm-john} \|\ident_{\R^n}\|_{N(X \to X^*)} = n. \end{equation}
A related consequence is the fact that for every $n$-dimensional normed space $X$, 
\begin{equation} \label{eq:nuclear-norm-id} \|\ident_X\|_{N(X \to X)} = n .\end{equation}

\subsection{Banach--Mazur distance} \label{subsection:Bm-distance}

Let $X$, $Y$ be Banach spaces. The Banach--Mazur distance between $X$ and $Y$ is given by
\[
\text{d}(X,Y)=\inf\lset \|U\|_{X\ra Y}\|U^{-1}\|_{Y\ra X}~:~U:X\ra Y\text{ linear bijection}\rset .
\]

For $1 \leq p \leq \infty$ and an integer $n \geq 1$, denote by $\ell_p^n$ the space $(\R^n,\|\cdot\|_p)$ where $\|\cdot\|_p$ is the usual $p$-norm. For an $n$-dimensional normed space $X$, we set
\[
d_X := \mathrm{d}(X,\ell_2^n).
\]
A standard estimate, which can be deduced from Theorem \ref{thm:john-lowner}, is
that 
\[ d_X \leq \sqrt{n} \]
for every $n$-dimensional normed space $X$.

\subsection{Spaces with enough symmetries} \label{subsec:enough-symmetries}

Let $X$ be a finite-dimensional normed space. We say that an invertible linear map $U : X \to X$ is a \emph{symmetry} of $X$ if it satisfies $\|Ux\|_X = \|x\|_X$ for every $x \in X$. The symmetries of $X$ form a compact group which we denote by $G(X)$. 

Following \cite[\S 16]{tomczak1989banach}, we say that $X$ \emph{has enough symmetries} if every linear map $T:X\ra X$ satisfying $T\circ U= U\circ T$ for each symmetry $U\in G(X)$ is a multiple of $\ident_X$. The class of spaces with enough symmetries includes the spaces $\ell_p^n$, and normed spaces obtained by equipping the space of matrices with a unitarily invariant norm such as the Schatten $p$-norms.

In a space with enough symmetries, there is (up to a positive scalar multiple) a unique inner product $\langle \cdot , \cdot \rangle$ which is invariant, i.e., such that $\langle U(x), U(y) \rangle = \langle x,y \rangle$ for every $x, y \in X$ and $U \in G(X)$. Since the inner product associated to either the John or the Loewner ellipsoid are invariant, it follows that for a space with enough symmetries, the John and Loewner ellipsoids are proportional. 

Let $X$ be a space with enough symmetries, and $\mathrm{d}U$ the Haar measure on $G(X)$. For any linear operator $T: X \to X$, we have
\begin{equation} \label{eq:twirling}
\int_{G(X)} U^{-1}\circ T\circ U \text{d}U = \Tr\lbr T\rbr \frac{\ident_X}{\text{dim}(X)},
\end{equation}
where $\Tr$ denotes the trace on $L(X,X)$.

\section{Basic properties of \texorpdfstring{$\tau_k$}{tau k} and \texorpdfstring{$\tau_\infty$}{tau infinity}} \label{sec:tau}

\subsection{Existence of \texorpdfstring{$\tau_\infty$}{tau infinity} and behaviour under transformations}

We now investigate more systematically the properties of the quantities $\tau_k$ and $\tau_\infty$ (see \eqref{eq:def-tau} for the definition). We first show that
\begin{equation} \label{eq:tau-basic-bound} \|T\|_{X \to Y} \leq \tau_k(T) \leq \|T\|_{N ( X \to Y)}. \end{equation}
The left inequality follows by restricting the supremum in \eqref{eq:def-tau} to tensors of the form $z=x^{\otimes k}$ for $x \in X$. To prove the right-hand side, consider a nuclear decomposition of the form $T(\cdot) = \sum_i x_i^*(\cdot) y_i$ with $x^*_1,\dots,x^*_n\in X^*$ and $y_1,\dots,y_n\in Y$. We have
\begin{eqnarray*} \left\| T^{\otimes k}(z) \right\|_{\pi_k(Y)} & \leq & \sum_{i_1,\dots,i_k}
\left| (x^*_{i_1} \otimes \dots \otimes x^*_{i_k})(z) \right| \cdot \|y_{i_1} \otimes \cdots \otimes y_{i_k}\|_{\pi_k(Y)} \\
& \leq & \sum_{i_1,\dots,i_k} \|x_{i_1}^*\| \dots \|x_{i_k}^*\| \cdot \|z\|_{\e_k(X)} \|y_{i_1} \| \dots \|y_{i_k}\| \\
& = & \left( \sum_i \|x_i^*\| \cdot \|y_i\| \right)^k \|z\|_{\e_k(X)},
\end{eqnarray*}
for every $z \in X^{\otimes k}$, and the result follows by optimizing over decompositions of $T$.

\begin{lem} \label{lem:subadditivity}
Let $X, Y$ be finite-dimensional normed spaces and $T:X \to Y$ a linear map. The limit of the sequence $(\tau_k(T))_{k \in \N}$ exists, and
\[ \tau_\infty(T):=\lim_{k \to \infty} \tau_k(T) = \sup_{k \geq 1} \tau_k(T). \]
\end{lem}

\begin{proof} For any integers $k_1, k_2 \geq 1$, we have
\begin{equation} \label{eq:tau-superadditive} (\tau_{k_1+k_2}(T))^{k_1+k_2} \geq (\tau_{k_1}(T))^{k_1}(\tau_{k_2}(T))^{k_2}, \end{equation}
as can be seen by restricting the supremum defining $\tau_{k_1+k_2}(T)$ to elements of the form $z_1 \otimes z_2$, for $z_1 \in X^{\otimes k_1}$, $z_2 \in X^{\otimes k_2}$. The conclusion follows by applying Fekete's lemma \cite{fekete1923verteilung} to the sequence $\lb t_k\rb_k$ given by $t_k = k\log (\tau_k(T))$. \end{proof}

By \eqref{eq:tau-superadditive} the inequality $\tau_k(T) \leq \tau_l(T)$ holds whenever $k$ divides $l$, but we have seen in Example \ref{example:hadamard} that the sequence $(\tau_k(T))_k$ is in general not monotonically increasing. An immediate consequence of \eqref{eq:tau-basic-bound} is the inequality
\begin{equation} \tau_{\iy}(T) \leq \|T\|_{N(X \to Y)}
\label{equ:tauInfUpperLower} .\end{equation}

The next lemma shows that the \name behaves nicely under adjoints.

\begin{lem} \label{lem:tau-duality}
Let $X$, $Y$ be finite-dimensional normed spaces, and $T \in L(X,Y)$. Then 
$\tau_k(T) = \tau_k(T^*)$ for any $k \in \N \cup \{\iy\}$.
\end{lem}

\begin{proof}
For any $k \geq 1$, we have $\|T^{\otimes k}\|_{\e_k(X) \to \pi_k(Y)} = 
\|(T^*)^{\otimes k}\|_{\e_k(Y^*) \to \pi_k(X^*)}$ by the duality between the injective and projective norms. The result follows.
\end{proof}

By Lemma \ref{lem:tau-duality}, a pair $(X,Y)$ of finite-dimensional normed spaces has the NTP (see Definition \ref{defn:NTP}) if and only if $(Y^*,X^*)$ has the NTP. \Names also satisfy the ideal property, in the following sense.

\begin{lem} \label{lem:ideal-property}
Let $X$, $X'$, $Y$, $Y'$ be finite-dimensional normed spaces and $T \in L(X,Y)$. For $A \in L(X',X)$, $B \in L(Y,Y')$ and every $k \in \N \cup \{\iy\}$ we have
\[ \tau_k(BTA) \leq \|B\| \tau_k(T) \|A\| .\]
\end{lem}

\begin{proof}
It suffices to prove the result for finite $k$. We combine the ideal property of the operator norm 
and the metric mapping property of the injective and projective norms to obtain
\[ \tau_k(BTA) \leq \|A^{\otimes k}\|^{1/k}_{\e_k(X') \to \e_k(X)} \tau_k(T) 
\|B^{\otimes k}\|^{1/k}_{\pi_k(Y) \to \pi_k(Y')} = \|A\| \tau_k(T) \|B\|. \qedhere \]
\end{proof}

\subsection{Upper bound by factorization through \texorpdfstring{$\ell_2$}{ell 2}}

The following lemma is a useful tool to find upper bounds on $\tau_\infty(T)$ by considering factorizations of the operator $T$ through Euclidean spaces. 

\begin{lem} \label{lem:UpperBound}
Let $X$ and $Y$ denote finite-dimensional normed spaces. For any $d\in\N$ and any pair of linear operators 
\[
Q_1:\ell^d_2\ra Y \quad\text{ and }\quad Q_2:X\ra \ell^d_2 ,
\]
we have
\[
\tau_\infty\lb Q_1\circ Q_2\rb \leq \|Q_1\circ Q^*_1\|^{\frac{1}{2}}_{N(Y^*\ra Y)}\|Q^*_2\circ Q_2\|^{\frac{1}{2}}_{N(X\ra X^*)} .
\]
\end{lem}

In Lemma \ref{lem:UpperBound}, the adjoint of $Q_1 = \sum \langle e_i,\cdot \rangle y_i$, where $(e_i)$ is the canonical basis of $\ell_2^d$ and $(y_i) \subset Y$, should be understood as $Q_1^* =
\sum y_i(\cdot) e_i : Y^* \to \ell_2^d$.

\begin{proof}
For any $k\in \N$, we have
\[ \|(Q_1 \circ Q_2)^{\otimes k}\|_{\e_k(X)\ra \pi_k(Y)} \leq \|Q_2^{\otimes k}\|_{\e_k(X) \to \mathrm{HS}_k(\ell_2^d)} \|Q_1^{\otimes k}\|_{\mathrm{HS}_k(\ell_2^d) \to \pi_k(Y)}
.\]
Fix an element $z \in X^{\otimes k}$ with $\|z\|_{\e_k(X)} \leq 1$, together with a nuclear decomposition
$Q^*_2\circ Q_2( \cdot) = \sum_{i} \tilde{x}_i(\cdot) x_i$ with $x_i,\tilde{x}_i\in X^*$. We find that
\begin{align*}
\|Q^{\otimes k}_2(z)\|^2_{\mathrm{HS}_k(\ell_2^d)} &= [(Q^*_1 \circ Q_1)^{\otimes k}(z)](z) \\
&\leq \sum_{i_1,\ldots ,i_k} |(x_{i_1}\otimes\cdots \otimes x_{i_k})(z)|\cdot |(\tilde{x}_{i_1}\otimes\cdots \otimes \tilde{x}_{i_k})(z)|\\
&\leq \lb\sum_i \|\tilde{x}_i\|_{X^*}\|x_i\|_{X^*}\rb^k.
\end{align*}

Optimizing over $z$ and over nuclear decompositions of $Q^*_2\circ Q_2$, we conclude  that
\[
\|Q^{\otimes k}_2\|_{\e_k(X) \to \mathrm{HS}_k(\ell_2^d)} \leq \|Q^*_2\circ Q_2\|^{k/2}_{N(X\ra X^*)}.
\]
By taking adjoints and using the previous inequality, we obtain 
\[ \|Q_1^{\otimes k}\|_{\mathrm{HS}_k(\ell_2^d) \to \pi_k(Y)} = 
\|(Q_1^*)^{\otimes k}\|_{\e_k(Y^*) \to \mathrm{HS}_k(\ell_2^d)}
\leq \|Q_1\circ Q_1^*\|^{k/2}_{N(Y^*\ra Y)}.
\]
Combining the previous bounds and taking the limit $k\ra \infty$ we have
\[
\tau_\infty\lb Q_1\circ Q_2\rb \leq \|Q_1\circ Q^*_1\|^{\frac{1}{2}}_{N(Y^*\ra Y)}\|Q^*_2\circ Q_2\|^{\frac{1}{2}}_{N(X\ra X^*)},
\]
and the proof is finished.
\end{proof}

\begin{rem}
Here is a reformulation of Lemma \ref{lem:UpperBound}: given a map $T : X \to Y$, we have
\[ \tau_\iy(T) \leq \inf \left\{ \left\| \sum x_i^*(\cdot) x_i^* \right\|^{1/2}_{N(X \to X^*)} 
\left\| \sum y_i(\cdot) y_i \right\|^{1/2}_{N(Y^* \to Y)}
\right\} \leq \|T\|_{N(X \to Y)}, \]
where the infimum is taken over decompositions $T(\cdot) = \sum x_i^*(\cdot) y_i$, with $(x_i^*) \subset X^*$ and $(y_i) \in Y$.
\end{rem}

We will often apply Lemma \ref{lem:UpperBound} by choosing $Q_1$ or $Q_2$ to be formal identity operators between normed spaces of the same dimension. In this case, techniques involving the John and/or Loewner ellipsoids can be used to obtain upper bounds that are often tight. We will use this approach in the following sections.

\section{\Names of normed spaces} \label{sec:rho}

In this section, we study the \name $\rho_\infty(X) = \tau_\infty(\mathrm{id}_X)$ of a normed space $X$. We first gather some elementary properties. 

\begin{prop} \label{prop:properties-rho}
For $n$-dimensional normed spaces $X$, $Y$ we have: 
\begin{enumerate}
    \item $\rho_k(X)=\rho_k(X^*)$ for every $k \in \N \cup \{\iy\}$.
    \item $\rho_k(X) \leq \mathrm{d}(X,Y) \rho_k(Y)$ for every $k \in \N \cup \{\iy\}$.
    \item $\rho_k(X) \leq n^{1-1/k}$ for every $k \in \N$, and $\rho_\infty(X)\leq n$.
\end{enumerate}
\end{prop}

\begin{proof} (1) is Lemma \ref{lem:tau-duality} applied to $T = \ident_X$. The inequality (2) can be deduced from the ideal property (Lemma \ref{lem:ideal-property}) by optimizing over linear bijections $U : Y \to X$. To obtain (3), consider an Auerbach basis $(x_i)$ for $X$, i.e. such that $\|x_i\|=\|x_i^*\|_*=1$, where $(x_i^*)$ denotes the basis of $X^*$ dual to $(x_i)$. Any $z \in X^{\otimes k}$ can be expanded as
\[ z = \sum_{i_1,\dots,i_{k-1}} x_1 \otimes \dots \otimes x_{k-1} \otimes y_{i_1\dots i_{k-1}} \]
for some $y_{i_1,\dots,i_{k-1}} \in X$. We have 
\[ \|z\|_{\pi_k(X)} \leq \sum_{i_1,\dots,i_{k-1}} \|y_{i_1\dots i_{k-1}}\|_* \leq n^{k-1} \max_{i_1,\dots,i_{k-1}} 
\|y_{i_1\dots i_{k-1}}\|_* \leq n^{k-1} \|z\|_{\e_k(X)}
\]
and the result follows.
\end{proof}

The quantity $\rho_2(X)$ (or rather, its square $\rho_2(X)^2$) was studied extensively in \cite{ALPSW20}, where it has been shown that
\[ n^{\frac 14 -o(1)} \leq \rho_2(X) \leq \sqrt{n} ,\]
as $n = \dim (X)$ tends to infinity. Both of these estimates are sharp since we have, for example, $\rho_2(\ell_2^n) = \sqrt{n}$ and $\rho_2(\ell_1^n) \leq (2n)^{\frac 14}$.

\subsection{\Names of Euclidean spaces} 

Our first result determines the \name of a Euclidean space.

\begin{thm} \label{thm:euclidean}
For every $n \geq 1$, we have $\rho_\iy(\ell_2^n)=n$.
\end{thm}

Our proof of Theorem \ref{thm:euclidean} uses the following lemma, which is an immediate extension of the $n=2$ case which appears in \cite[Proposition 8.28]{aubrun2017alice}. Since it is based on a standard random construction, we omit the proof here. We denote by $\|\cdot\|_{\mathrm{HS}}$ the Hilbert--Schmidt norm on $(\ell_2^d)^{\otimes k}$.

\begin{lem} \label{lem:random-entanglement}
Given integers $n, k \geq 2$, there exists a tensor $z \in (\ell_2^n)^{\otimes k}$ such that 
$\|z\|_{\mathrm{HS}}=1$ and
\[ \|z\|_{\e_k(\ell^n_2)} \leq \frac{C_n \sqrt{k \log k}}{n^{k/2}}, \]
where $C_n$ is a constant which does not depend on $k$.
\end{lem}

An improved version of Lemma \ref{lem:random-entanglement} can be proved in the real case: There is a tensor $z$ in the real space $(\ell_2^n)^{\otimes k}$ such that $\|z\|_{\mathrm{HS}} = 1$ and $|z\|_{\e_k(\ell^n_2)} \leq C_n/n^{k/2}$ (see~\cite{aubrun2017alice,paperA}). It seems to be unknown whether such an improvement is possible in the complex case, even for $n=2$ (see \cite[Problem 8.27]{aubrun2017alice}).

\begin{proof}[Proof of Theorem \ref{thm:euclidean}]
By Proposition \ref{prop:properties-rho}, we have $\rho_{\infty}(\ell_2^n)\leq n$ and we only need to prove that $\rho_{\infty}(\ell_2^n)\geq n$. Fix $k \geq 1$. The Hilbert--Schmidt norm on $(\ell_2^n)^{\otimes k}$ satisfies the inequality $\|z\|_{\mathrm{HS}}^2 \leq \|z\|_{\pi_k(\ell^n_2)}\|z\|_{\e_k(\ell^n_2)}$ for every $z \in (\ell_2^n)^{\otimes k}$. Therefore, we have
\[ 
\rho_k(\ell_2^n) \geq \sup_{z \in (\ell_2^n)^{\otimes k}} \left( \frac{\|z\|_{\mathrm{HS}}}{\|z\|_{\e_k(\ell^n_2)}} \right)^{2/k}.
\] 
Using Lemma \ref{lem:random-entanglement} and taking the limit $k \to \infty$ shows that $\rho_k(\ell_2^n)\geq n$.
\end{proof}

\subsection{Proof of Theorem~\ref{thm:Main1}} \label{sec:proof-th1}

Let $X$ be a $n$-di\-men\-sion\-al normed space. Since $\rho_\infty(\ell_2^n)=n$, we obtain the lower bound
\begin{equation}\label{equ:lowerBoundOnRhoInf}
\rho_{\infty}(X) \geq \frac{ \rho_{\infty}(\ell_2^n)}{\mathrm{d}_X} \geq \sqrt{n} 
\end{equation}
using Proposition \ref{prop:properties-rho} and Theorem \ref{thm:euclidean}.

It remains to show that $\rho_{\infty}(X)<n$ if we assume that $X$ is not Euclidean. Without loss of generality we may assume that $X=(\R^n,\|\cdot\|)$ such that $B_{\ell^n_2}$ is the Loewner ellipsoid of $X$. Identifying $X^*$ with $(\R^n,\|\cdot\|_{*})$ such that $\|x\|_*=\sup\lset |\braket{x}{y}|:\|y\|\leq 1\rset$ implies that $\|\cdot\|_* \leq \|\cdot\|_2$. Since $X^*$ is not Euclidean, there exists $x_1 \in \R^n$ such that $\|x_1\|_2=1$ and 
$\|x_1\|_*<1$. Next, we complete $\{x_1\}$ to an orthonormal basis $\lb x_1,\ldots ,x_n\rb$ of $\ell_2^n$ and note that $\|x_i\|_* \leq 1$ for every $i$. Since $\ident_{\R^n} = \sum^n_{i=1} x_i \braket{x_i}{\cdot}$, we have $\|\ident_{\R^n}\|_{N(X\ra X^*)}\leq \sum^n_{i=1} \|x_i\|^2_{*}<n$. Furthermore, by \eqref{eq:nuclear-norm-john} we have $\|\ident_{\R^d}\|_{N(X^*\ra X)} = n$. Using Lemma \ref{lem:UpperBound} for $Q_1=Q_2=\ident_{\R^n}$, we find that 
\[
\rho_\infty(X)= \tau_\infty(\ident_{\R^n}) \leq \|\ident_{\R^n}\|^{1/2}_{N(X\ra X^*)}\|\ident_{\R^n}\|^{1/2}_{N(X^*\ra X)} < n.
\] 
This finishes the proof.

\subsection{Proof of Theorem \ref{thm:enough-symmetries}} \label{sec:proof-th2}

We start with an easy lemma: If the John and Loewner ellipsoids are proportional, they realize the infimum in the Banach--Mazur distance to the Euclidean space.

\begin{lem}
\label{lem:john-loewner-distance}
Let $X$ be an $n$-dimensional normed space with John ellipsoid $\mathcal{E}$ and with Loewner ellipsoid $\alpha \mathcal{E}$ for some $\alpha \geq 1$. Then $\alpha = \mathrm{d}_X$.
\end{lem}

\begin{proof} The inequality $\mathrm{d}_X\leq \alpha$ is immediate. Conversely, if an ellipsoid $\mathcal{F} \subset X$ and a number $\beta \geq 1$ satisfy $\mathcal{F} \subset B_X \subset \beta \mathcal{F}$, then by the definition of the John and Loewner ellipsoids,
\[ \mathrm{vol}(\mathcal{F}) \leq \mathrm{vol}(\mathcal{E}), \ \ \mathrm{vol}(\beta \mathcal{F}) \geq \mathrm{vol}(\alpha \mathcal{E}) \]
from which we infer that $\beta \geq \alpha$.
\end{proof}

\begin{proof}[Proof of Theorem \ref{thm:enough-symmetries}]
We already noticed in \eqref{equ:lowerBoundOnRhoInf} that the lower bound $\rho_{\infty}(x) \geq n/\mathrm{d}_X$ holds for every $n$-dimensional space $X$. Set $\alpha = \mathrm{d}_X$ and assume without loss of generality that $X=(\R^n,\|\cdot\|)$, with $B_{\ell^n_2}$ being the John ellipsoid for $X$. By Lemma \ref{lem:john-loewner-distance}, $\alpha B_{\ell^n_2}$ is the Loewner ellipsoid of $X$.

Since $\alpha B_{\ell^n_2}$ is the John ellipsoid of $X^*=(\R^n,\|\cdot\|_*)$ with $\|x\|_*=\sup\lset |\braket{x}{y}/\alpha^2|:\|y\|\leq 1\rset$, we can apply Theorem \ref{thm:john-lowner} to obtain a decomposition
\[
\ident_{\R^n}=\sum_{i}c_i v_i\braket{v_i}{\cdot},
\]
with $v_i\in\R^n$ satisfying $\|v_i\|_2=\alpha$ and $\|v_i\|=1$. Taking the trace shows that $\sum_i c_i = n/\alpha^2$ and we find that 
\[
\|\ident_{\R^d}\|_{N(X^*\ra X)} \leq \frac{n}{\alpha^2}.
\]
Together with \eqref{eq:nuclear-norm-john} (using that $B_{l^n_2}$ is the John ellipsoid of $X$) and by Lemma \ref{lem:UpperBound} for $Q_1=Q_2=\ident_{\R^n}$ this implies
\[ \rho_\infty(X) \leq \|\ident_{\R^n}\|^{1/2}_{N(X \to X^*)}\|\ident_{\R^n}\|^{1/2}_{N(X^* \to X)} \leq \frac{n}{\alpha}, \]
as needed.
\end{proof}

As a corollary to Theorem \ref{thm:enough-symmetries}, we compute the \name of many natural examples of finite-dimensional normed spaces:

\begin{cor}
For $n \geq 1$ and $1 \leq p \leq \infty$,
\begin{itemize}
    \item the \name of the space $\ell_p^n$ equals
    \[ \rho_\infty(\ell^n_p) = n^{1-\left| \frac 12 - \frac 1p \right|} \]
    \item the \name of the space $S_p^n$, defined as the space of $n \times n$ matrices equipped with Schatten $p$-norm $\|X\|_{S_p}=(\Tr\lbr |X|^p\rbr)^{1/p}$, equals
    \[ \rho_\infty(\mathcal{S}^n_p) = n^{2 -
\left| \frac 12 - \frac 1p \right|}. \]
\end{itemize}
\end{cor}

\section{\Names of operators} \label{sec:tau=nuclear-for-euclidean}

\subsection{Proof of Theorem \ref{thm:main-nuclear-norm}}\label{sec:ProofofMain-nuclear-norm}
We start by restating Theorem \ref{thm:main-nuclear-norm}:

\begin{thm}[Theorem \ref{thm:main-nuclear-norm}, restated]
\label{thm:tauRedIdentity}
Consider $n\in\N$ and a finite-dimensional normed space $Y$. For every operator $T \in L(\ell_2^n,Y)$ or $T \in L(Y,\ell_2^n)$, we have
\begin{equation} \label{eq:tauRedIdentity}
\tau_{\infty}(T) = \|T\|_{N}.
\end{equation}
\end{thm}

\begin{proof}
The inequality $\tau_\infty(T) \leq \|T\|_N$ holds by \eqref{equ:tauInfUpperLower}.
Consider $T \in L(\ell_2^n,Y)$, and let $Q \in L(Y,\ell_2)$ such that $\|Q\| \leq 1$. Let $\mathrm{d}U$ be the Haar measure on the orthogonal group $O(n)$. Define $R : \ell_2^n \to \ell_2^n$ by
\begin{equation}
R := \int_{O(n)} U^{-1}\circ Q\circ T\circ U \text{d}U .
\label{equ:Rdef}
\end{equation}
By \eqref{eq:twirling}, we have $R = \alpha \ident_{\ell_2^n}$ with $\alpha = \Tr [QT]/n$.
On the other hand, for every $k\in\N$, we have
\[
R^{\otimes k} = \int_{O(n)} \cdots \int_{O(n)} (U_1^{-1} \otimes \cdots \otimes U_k^{-1}) \circ (Q \circ T)^{\otimes k} \circ (U_1 \otimes \cdots \otimes U_k) \mathrm{d} U_1 \dots \mathrm{d}U_k.
\]

Using the triangle inequality, the ideal property of the operator norm, and taking the $k$th root shows that $\tau_k(R) \leq \tau_k(Q \circ T)$. Taking the limit $k \to \infty$ gives
\[ \tau_\infty(R) \leq \tau_\infty(Q \circ T) \leq \tau_\infty(T). \]
We have $\tau_\infty(R) = \alpha \tau_\infty(\ident_{\ell_2^n}) = \Tr [QT]$, using Theorem 
\ref{thm:euclidean}. We proved that
\begin{equation} \label{eq:trace-duality-in-proof} \sup_{\|Q\| \leq 1} \Tr \lbr QT\rbr \leq \tau_\infty(T) \end{equation}
By trace duality (cf \eqref{eq:trace-duality}), the left-hand side of \eqref{eq:trace-duality-in-proof} is exactly the nuclear norm $\|T\|_N$, completing the proof.

If $T \in L(Y,\ell_2^n)$, the proof is completely similar by considering the supremum of $\Tr \lbr QT\rbr$ over $Q : \ell_2^n \to Y$ such that $\|Q\| \leq 1$.
\end{proof}

The following corollary establishes a general lower bound on $\tau_\infty$ in terms of the nuclear norm. This also generalizes the bound from \eqref{equ:lowerBoundOnRhoInf} on $\rho_\infty$.

\begin{cor}\label{cor:tauInftyLowerBound}
Let $X, Y$ be finite dimensional normed spaces. We have
\begin{equation} 
\tau_{\infty}(T) \geq \frac{\|T\|_{N}}{\min(d_X,d_Y)} ,
\end{equation}
for every operator $T \in L(X,Y)$.
\end{cor}

\begin{proof}
Assume that $\dim X=n$ and that $d_X\leq d_Y$ (if $d_Y>d_X$, then the proof works in the same way). Consider a bijection $U \in L(\ell_2^n,X)$. We have 
\[ \|T\|_N \leq \|TU\|_N \|U^{-1}\| = \tau_\infty(TU) \|U^{-1}\| \leq \tau_\infty(T) \|U\| \cdot \|U^{-1}\|\]
where the equality follows from Theorem \ref{thm:tauRedIdentity}, and the inequalities from the ideal property of $\tau_{\infty}$ and $\|\cdot\|_N$. The result follows by taking the infimum over $U$.
\end{proof}

\subsection{Normed spaces without the nuclear tensorization property}\label{sec:SpacesWithoutNTP}
Theorem \ref{thm:Main1} shows that for every non-Euclidean space $X$ the pair $(X,X)$ does not have the nuclear tensorization property (NTP, see Definition \ref{defn:NTP}). Moreover, Theorem \ref{thm:main-nuclear-norm} shows that $(X,Y)$ has the NTP when $X$ or $Y$ is Euclidean, and one may ask whether the converse holds. In the case of spaces of the same dimension, the following criterion shows that a large class of examples fails the NTP.

\begin{prop} \label{prop:john-avoids-loewner}
Let $X$, $Y$ be finite-dimensional normed spaces of the same dimension. Let $\mathcal{E}$ be the John ellipsoid of $X$ and $\mathcal{F}$ be the Loewner ellipsoid of $Y$. Assume that $T : X \to Y$ is a linear map such that $T(\mathcal{E})=\mathcal{F}$ and such that, whenever $x$ is a John contact point of $X$, $T(x)$ is not a Loewner contact point of $Y$. Then $\tau_\iy(T) < \|T\|_{N(X \to Y)}$ and therefore $(X,Y)$ fails the NTP.
\end{prop}

\begin{proof}
Without loss of generality, assume that $X=(\R^n,\|\cdot\|_X)$, $Y=(\R^n,\|\cdot\|_Y)$, $\mathcal{E} = \mathcal{F} = B_2^n$ and $T = \ident_{\R^n}$. Applying Lemma \ref{lem:UpperBound} for $Q_1=Q_2=\ident_{\R^n}$ implies
\[ \tau_{\iy}(T) \leq \|\ident_{\R^n}\|^{1/2}_{N(X \to X^*)}
\|\ident_{\R^n}\|^{1/2}_{N(Y^* \to Y)} \leq n, \]
by using \eqref{eq:nuclear-norm-john} twice. 

Let $x \in \R^n$ such that $\|x\|_2 = 1$. We have $\|x\|_X \leq 1 \leq \|x\|_Y$. Since $x$ cannot be both a John contact point of $X$ and a Loewner contact point of $Y$, it follows that $\|x\|_X < \|x\|_Y$ and therefore $\|\ident_{\R^n}\|_{Y \to X} < 1$. By trace duality, we have
\[ n = \Tr (\ident_{\R^n}) \leq \|\ident_{\R^n}\|_{Y \to X} \|\ident_{\R^n}\|_{N(X \to Y)} \]
and therefore $\|T\|_{N(X \to Y)} > n \geq \tau_\infty(T)$.
\end{proof} 

Proposition \ref{prop:john-avoids-loewner} can be applied when $X$ (resp., $Y$) has few John (resp., Loewner) contact points. For example, real spaces with polyhedral unit balls have finitely many John and Loewner contact points; in this case the hypothesis of Proposition \ref{prop:john-avoids-loewner} is satisfies for a generic map $T$. 

We now consider the case of a pair $(\ell_\iy^2,Y)$.

\begin{prop} \label{prop:conjecture-square}
Let $Y$ be a finite-dimensional normed space. The pair $(\ell_\infty^2,Y)$ has the NTP if and only if $Y$ is Euclidean.
\end{prop}

\begin{proof}
It suffices to show that $(\ell_\infty^2,Y)$ fails the NTP is we assume that $Y$ is not Euclidean.
We use a classical characterization of Euclidean spaces \cite{Schoenberg52}: a norm $\|\cdot\|$ is Euclidean if and only if the inequality $\|x+y\|^2 + \|x-y\|^2 \geq 4$ holds for every unit vectors $x,y$. Since $Y$ is not Euclidean, it contains unit vectors $x$, $y$ such that $\|x+y\|^2 + \|x-y\|^2 < 4$.

Consider the operator $T : \ell_\infty^2 \to Y$ given by $T(a,b)=a x + b y$. By~\cite[Proposition 8.7]{tomczak1989banach}, we have $\|T\|_{N(\ell_\infty^2 \to Y)} = \|x\|+\|y\|= 2$. We now prove that
\begin{equation} \label{eq:square-2}
 \tau_\infty(T) \leq \left(\|x+y\|^2 + \|x-y\|^2 \right)^{1/2} < 2. \end{equation}
To prove \eqref{eq:square-2}, write $T=Q_1 \circ Q_2$, where $Q_1 : \ell_2^2 \to Y$ is given by $Q_1(a,b) = a(x+y)+b(x-y)$ and $Q_2 : \ell_{\infty}^2 \to \ell_2^2$ is given by $Q_2(a,b)=\frac 12 ( a+b,a-b)$. By Lemma \ref{lem:UpperBound}, we have
\[ \tau_\iy(T) \leq \|Q_1 \circ Q_1^*\|_{N(Y^* \to Y)}^{1/2} 
\|Q_2^* \circ Q_2\|_{N(\ell_\infty^2 \to \ell_1^2)}^{1/2} \]
Since $Q_2^* \circ Q_2 = \frac{1}{2}\ident_{\R^2}$, we have $\|Q_2^* \circ Q_2\|_{N(\ell_\infty^2 \to \ell_1^2)}=1$. The map $Q_1 \circ Q_1^*$ can be decomposed as $(x+y)(\cdot)(x+y) + (x-y)(\cdot)(x-y)$; this decomposition gives the bound $\|Q_1 \circ Q_1^*\|_{N(Y^* \to Y)} \leq \|x+y\|^2 + \|x-y\|^2$. We obtained the inequality $\tau_{\iy}(T) < 2 = \|T\|_{N(\ell_\iy^2 \to Y)}$, showing that $(\ell_\iy^2,Y)$ fails the NTP.
\end{proof}

Let $Y$ be a normed space. A subspace $Y' \subset Y$ is said to be \emph{$1$-complemented} if there is a projection $P:Y \to Y'$ with $\|P\| = 1$. The next lemma shows that the NTP is inherited by complemented subspaces.

\begin{lem} \label{lem:complementation}
Let $X$, $Y$ be finite-dimensional normed spaces and $X' \subset X$, $Y' \subset Y$ be $1$-complemented subspaces. If $(X,Y)$ has the NTP, then $(X',Y')$ has the NTP.
\end{lem}

\begin{proof}
Let $\iota_1 : X' \to X$, $\iota_2 : Y' \to Y$ be the inclusion maps, and $P_1 : X \to X'$, $P_2~:~Y \to Y'$ be projections of norm $1$. Let $T \in L(X',Y')$ a linear map and set $S = \iota_2 \circ T \circ P_1 \in L(X,Y)$. The ideal property of the nuclear norm and the NTP for $(X,Y)$ imply that
\[ \|T\|_{N (X' \to Y')} = \|P_2 \circ S \circ \iota_1\|_{N (X' \to Y')} \leq  \|S \|_{N (X \to Y)}
= \tau_\infty(S). \]
Similarly, the ideal property of the \name (see Lemma \ref{lem:ideal-property}) implies that $\tau_\iy(S) \leq \tau_\iy(T)$. We conclude that $\tau_\iy(T) = \|T\|_{N(X' \to Y')}$ and the result follows.
\end{proof}

Let $X$ be a finite-dimensional normed space containing a subspace isometric to $\ell_\iy^2$ (this includes in particular the spaces $\ell_\iy^n$ and, in the real case, the spaces $\ell_1^n$). Such a subspace is necessarily $1$-complemented (this follows from the Hahn--Banach theorem). It follows from Proposition \ref{prop:conjecture-square} and Lemma \ref{lem:complementation} that $(X,Y)$ fails the NTP whenever $Y$ is not Euclidean.

Our methods also imply that the $\ell_p$-spaces over the reals, denoted by $\ell_p^n(\R)$, only have the NTP in the Euclidean case.

\begin{prop}
Let $m,n \geq 2$ and $p,q \in [1,\iy]$. If $(\ell_p^m(\R),\ell_q^n(\R))$ has the NTP, then $p=2$ or $q=2$.
\end{prop}

\begin{proof}
By Lemma \ref{lem:complementation}, we may assume that $m=n=2$ ($\ell_p^n$ contains a $1$-com\-ple\-ment\-ed subspace isometric to $\ell_p^2$). If $p \neq 2$, the John contact points of $\ell_p^2$ (which coincide with the Loewner contact points of $\ell_{p*}^2$ if $1/p+1/p^*=1$), are the following $4$ points (up to normalization)
\[ \begin{cases} (\pm 1,0) \textnormal{ or } (0,\pm 1) & \textnormal{ if } p > 2  \\
(\pm 1,\pm 1) & \textnormal{ if } p < 2.
\end{cases} \]
If both $p \neq 2$ and $q \neq 2$, the fact that the pair $(\ell_p^2,\ell_q^2)$ does not have the NPT follows from Proposition \ref{prop:john-avoids-loewner}, by choosing $T : \R^2 \to \R^2$ to be a rotation of angle not multiple of $\pi/4$. The result follows.
\end{proof}

Some natural pairs of spaces are not covered by the criterion of Proposition \ref{prop:john-avoids-loewner}. Examples of such pairs, for which we do not know whether they have the NTP, are 
\begin{enumerate}
\item the pair $(X,X^*)$, where $X$ is the space $\R^2$ equipped with the norm \[ \max(\|\cdot\|_2,(1+\e)\|\cdot\|_{\iy}) \] for some small $\e>0$, 
\item the pair $(X,X^*)$, where $X$ is the space $\ell_1^2 \otimes_\pi \ell_2^2$,
\item in the complex case, the pair $(\ell_1^2,\ell_\infty^2)$ -- the previous example is obtained by forgetting the complex structure.
\end{enumerate}
Note that all these examples have enough symmetries. 

\section{Further questions about \names and some answers}\label{sec:FurtherQuestions}

\subsection{Is the \name continuous, or even a norm?}\label{sec:Norm}

One may wonder whether $\tau_k$ is a norm on the space $L(X,Y)$. Here is an example showing that this is not the case in full generality.

\begin{example} 
Over the real field, consider the operators $S$, $T : \ell_1^2 \to \ell_1^2$ given by the following matrices
\[ S = \begin{pmatrix} 1 & 1/3 \\ 1/3 & 1 \end{pmatrix}, \ \ T = \begin{pmatrix} 1/2 & 0 \\ 0 & 1 \end{pmatrix}. \]
One computes that $\tau_2(S) = \sqrt{2}$, $\tau_2(S) = \frac{3}{2\sqrt{2}}$ and $\tau_2(S+T) =  \sqrt{\frac{155}{24}} $, so that
\[ 2.54\ldots \approx \tau_2 (S+T)  >  \tau_2(S) + \tau_2(T) \approx 2.47\ldots \]
To obtain these values, we used the fact that the extreme points of the unit ball for the norm $\eps_2(\ell_1^2)$ are given by $\pm e_i \otimes e_j$ and $\frac 12 (\eta_1 e_1 \otimes e_1 + \eta_2 e_1 \otimes e_2 + \eta_3 e_2 \otimes e_1 + \eta_4 e_2 \otimes e_2)$, where $(\eta_i)$ is a permutation of either $(1,1,1,-1)$ or $(1,-1,-1,-1)$.
\end{example}

Despite this counterexample, it turns out that $\tau_2$ is a norm in the special case of $L(\ell_2^m,\ell_2^n)$. We have the following proposition, observed without proof in~\cite{KJohn1}

\begin{prop}
Let $m,n$ be integers. For $T \in L(\ell_2^m,\ell_2^n)$, $\tau_2(T)$ coincides with the Hilbert--Schmidt norm given by $\|T\|_{\mathrm{HS}} = \left( \Tr \lbr TT^*\rbr \right)^{1/2}$.
\end{prop}

\begin{proof}
Consider $T \in L(\ell_2^m,\ell_2^n)$. If $U \in L(\ell_2^m,\ell_2^m)$ and $V \in L(\ell_2^n,\ell_2^n)$ are isometries, we have $\tau_2(UTV)=\tau_2(T)$ by the ideal property of $\tau_2$ (Lemma \ref{lem:ideal-property}). Using the singular value decompositions, it suffices to prove the result when $T$ is diagonal, i.e., of the form $T = \sum_{i=1}^{\min(m,n)} \lambda_i e_ie_i^*$ for $\lambda_i \geq 0$. We have
then
\begin{align*} \tau_2(T)^2 & = \sup \left\{ \left| \langle (T\otimes T)(x),y \rangle \right| \, :\, \|x\|_{\ell_2^m \otimes_\e \ell_2^m} \leq 1,\ \|y\|_{\ell_2^n \otimes_\e \ell_2^n} \leq 1 \right\} \\
 & = \sup \left\{ \left| \Tr \lbr TXTY^t\rbr \right| \, :\, \|X\|_{\ell_2^m \to \ell_2^m} \leq 1,\ \|Y\|_{\ell_2^n \to \ell_2^n} \leq 1 \right\} \end{align*}
after identifying tensors with operators. For $X$, $Y$ as above, we have 
\[ | \Tr \lbr TXTY^t\rbr | \leq \|T\|_{\mathrm{HS}}\|XTY^t\|_{\mathrm{HS}} \leq \|T\|_{\mathrm{HS}}^2 \] with equality when $X=\ident_{\ell_2^m}$, $Y = \ident_{\ell_2^n}$, proving the result.
\end{proof}

We do not know whether $\tau_k$ is a norm on $L(\ell_2^m,\ell_2^n)$ for any other $k\in\N$. It would also be interesting to decide whether the tensor radius $\tau_{\infty}$ is a norm on $L(X,Y)$ (that question also appears in \cite{KJohn4}). By Theorem \ref{thm:main-nuclear-norm} the answer to this question is positive when $X$ or $Y$ is Euclidean. However, in general we do not know the answer. We will now show a weaker version of the triangle inequality, which implies that the tensor radius is a continuous function.

\begin{prop} \label{prop:tau-continuous}
For finite-dimensional normed spaces $X$, $Y$ and $S, T \in L(X,Y)$, we have
\[ \tau_\infty(S+T) \leq \tau_\infty(S) + \|T\|_N .\]
\end{prop}

Using Corollary \ref{cor:tauInftyLowerBound}, the previous proposition implies:

\begin{cor}[Weak triangle inequality]
Let $X$, $Y$ be finite-dimensional normed spaces and $S, T \in L(X,Y)$. Then
\[ \tau_\infty(S+T) \leq \tau_\infty(S) + \min\lb d_X,d_Y\rb\tau_\infty(T) .\]
\end{cor}

We first prove a lemma:

\begin{lem}\label{lem:EasyLemNormsRank1}
For $T \in L(X_1,Y_1)$, $y\in Y_2$, and $x^*\in X_2^*$, we have 
\[
\|T\otimes yx^*\|_{X_1\otimes_\e X_2 \to Y_1\otimes_\pi Y_2} = \|T\|_{X_1\ra Y_1}\|x^*\|_{X_2^*}\| y\|_{Y_2}.
\] 
\end{lem}

\begin{proof}
For any $z\in X_1\otimes X_2$ we have 
\[
\|(T\otimes yx^*)(z)\|_{Y_1\otimes_\pi Y_2} = \|T(z_x)\|_{Y_1}\|y\|_{Y_2},
\]
for $z_x= (\ident_{X_1}\otimes x^*)(z)\in X_1$, by the metric mapping property of $\pi$. Since $\|z_x\|_{X_1}\leq \|x^*\|_{X_2^*}\|z\|_{X_1\otimes_\e X_2}$, we conclude that 
\[
\|(T\otimes yx^*)(z)\|_{Y_1\otimes_\pi Y_2} \leq \|T\|_{X_1\ra Y_1} \|x^*\|_{X_2^*}\|y\|_{Y_2}\|z\|_{X_1\otimes_\e X_2},
\]
showing one direction of the identity in the lemma. The other direction follows by inserting $z=z_1\otimes z_2$ for a suitable choice of $z_1$ and $z_2$. 
\end{proof}

\begin{proof}[Proof of Proposition \ref{prop:tau-continuous}]
Let $T=\sum_i y_ix^*_i$ denote a nuclear decomposition. For $i,j\in \N$, Lemma \ref{lem:EasyLemNormsRank1} implies 
\begin{align*}
\|& S^{\otimes (k-j)} \otimes T^{\otimes j} \|_{ \e_{k-j}(X) \otimes_\e \e_j(X) \to \pi_{k-j}(Y) \otimes_\pi \pi_j(Y)} \\
&\leq \sum_{i_1,\ldots ,i_j} \| S^{\otimes (k-j)} \otimes (y_{i_1}x^*_{i_1}) \otimes \cdots \otimes (y_{i_j}x^*_{i_j})\|_{ \e_{k-j}(X) \otimes_\e \e_j(X) \to \pi_{k-j}(Y) \otimes_\pi \pi_j(Y)} \\
&\leq \| S^{\otimes (k-j)}\|_{\e_{k-j}(X) \to \pi_{k-j}(Y)} \lb\sum_{i}\|y_i\|_Y\|x^*_i\|_{X^*}\rb^j .
\end{align*}
Using the associativity and commutativity of the injective and projective tensor products, we may write
\begin{align*}
\tau_k(S+T)^k &=  \| (S+T)^{\otimes k} \|_{\e_k(X) \to \pi_k(Y)}   \\
& \leq \sum_{j=0}^k \binom{k}{j} \| S^{\otimes (k-j)} \otimes T^{\otimes j} \|_{ \e_{k-j}(X) \otimes_\e \e_j(X) \to \pi_{k-j}(Y) \otimes_\pi \pi_j(Y)} \\
& \leq \sum_{j=0}^k \binom{k}{j}  \|S^{\otimes (k-j)} \|_{ \e_{k-j}(X)\to \pi_{k-j}(Y)} \lb\sum_{i}\|y_i\|_Y\|x^*_i\|_{X^*}\rb^{j} \\
& \leq \left( \tau_\iy(S) + \sum_{i}\|y_i\|_Y\|x^*_i\|_{X^*} \right)^k .
\end{align*}
We are done by taking $k \to \iy$ and optimizing over nuclear decompositions of $T$.
\end{proof}

\subsection{Is the \name multiplicative?}\label{sec:submultiplicativity}

Another natural question is whether the quantities $\tau_\infty$ and $\rho_\infty$ are multiplicative under taking injective or projective tensor products. The examples for which we know $\rho_\infty$ immediately give a counterexample:
\[
\rho_\infty(\ell^d_2\otimes_\pi \ell^d_2) = \rho_\infty(S^d_1) = d^{3/2}  < d^2 = \rho_\infty(\ell^d_2)^2 .
\]
However, it turns out that $\tau_\infty$ and $\rho_\infty$ are submultiplicative in the following sense:
\begin{prop}
Let $X, Y$ be finite dimensional normed spaces and $T \in L(X,Y)$. Consider a crossnorm $\otimes_\alpha$ on $X \otimes X$ and a crossnorm $\otimes_\beta$ on $Y \otimes Y$. Then
\[
\tau_\infty\lb \lbr T\otimes T:X\otimes_\alpha X \ra Y\otimes_\beta Y\rbr \rb \leq \tau_\infty(T)^2 ,
\]
In particular, we have
\[
\rho_\infty(X\otimes_\alpha X)\leq \rho_\infty(X)^2 ,
\]
for any tensor norm $\otimes_\alpha$.
\label{prop:Submultiplicativity}
\end{prop}

We start with the following lemma, whose proof follows by using the metric mapping property of $\pi_k$ and $\e_k$.

\begin{lem}
For finite-dimensional normed spaces $X$ and $Y$, we have 
\[
\|\cdot\|_{\pi_k(X\otimes_\alpha Y)} \leq \|\cdot\|_{\pi_k(X\otimes_\pi Y)} \text{ and } \|\cdot\|_{\e_k(X\otimes_\alpha Y)} \geq \|\cdot\|_{\e_k(X\otimes_\e Y)},
\]
for every crossnorm $\otimes_\alpha$ on $X \otimes X$.
\label{lem:MixedTP}
\end{lem}

We can now present the proof of Proposition \ref{prop:Submultiplicativity}:

\begin{proof}[Proof of Proposition \ref{prop:Submultiplicativity}]
Let $T:X\ra Y$ denote a linear operator and consider Banach space tensor products $\otimes_\alpha$ and $\otimes_\beta$. In the following, we consider the operator $T\otimes T:X\otimes_\alpha X\ra Y\otimes_\beta Y$. For every $k\in \N$ we have
\begin{align*}
\tau_k(T\otimes T)^k &= \sup_{z\in (X\otimes X)^{\otimes k}} \frac{\|(T\otimes T)^{\otimes k}(z)\|_{\pi_k(Y\otimes_\alpha Y)}}{\|z\|_{\e_k(X\otimes_\beta X)}}\\
&\leq \sup_{z\in (X\otimes X)^{\otimes k}} \frac{\|(T\otimes T)^{\otimes k}(z)\|_{\pi_k(Y\otimes_\pi Y)}}{\|z\|_{\e_k(X\otimes_\e X)}} = \tau_{2k}(T)^{2k}.
\end{align*}
Taking the $k$th root and the limit $k\ra \infty$ finishes the proof.
\end{proof}

Note that $X\otimes_\alpha X$ has enough symmetries for any tensor norm $\alpha$ and any $X$ with enough symmetries (see~\cite[p.~62]{junge2021geometry}). Combining this fact with the previous proposition and with Theorem \ref{thm:enough-symmetries} leads to an estimate of independent interest:

\begin{cor} \label{cor:distance-tensor-product}
Let $X$ be a $n$-dimensional normed space with enough symmetries and $\alpha$ a tensor norm. Then
\[ \mathrm{d}_{X \otimes_\alpha X} \geq \mathrm{d}_X^2. \]
\end{cor}

A natural question is whether the assumption that $X$ has enough symmetries can be removed. More generally, we could ask if for finite-dimensional normed spaces $X$ and $Y$ the inequality
\[ \mathrm{d}_{X \otimes_\alpha Y} \geq \mathrm{d}_X \mathrm{d}_Y \]
holds for every tensor norm $\alpha$.

\subsection{When is the \name minimal?}\label{sec:Minimality}

Another natural problem is to determine the $n$-dimensional spaces for which the \name equals $\sqrt{n}$, the minimal possible value. This is achieved by both $\ell_1^n$ and $\ell_{\infty}^n$, but there are many more examples: The construction from \cite[Example 1]{ATT05} produces a continuum of normed spaces $(\R^n,\|\cdot\|)$ for which $B_2^n$ is the John ellipsoid and $\sqrt{n}B_2^n$ is the Loewner ellipsoid; by Theorem \ref{thm:enough-symmetries}, each such space has a \name equal to $\sqrt{n}$. Other examples can be produced by using Corollary \ref{cor:distance-tensor-product}: For any $n\in \N$ and any tensor norm $\alpha$, consider the space $X = \ell_1^n \otimes_\alpha \ell_1^n$. The space $X$ has enough symmetries and satisfies $\mathrm{d}_X= n = \sqrt{\mdim(X)}$, and we conclude by Theorem \ref{thm:enough-symmetries} that $\rho_\infty(X)=\sqrt{\mdim(X)}$. We should also note that Corollary \ref{cor:distance-tensor-product} produces many examples of spaces with maximal distance to Euclidean in this way.      

Another question concerns the lower bound
\[ \rho_{\infty}(X) \geq \frac{n}{\mathrm{d}_X}, \]
which holds for every $n$-dimensional normed space $X$: Are there spaces for which this inequality is strict?

\section{Infinite dimensions} \label{sec:infinite}

While we focused exclusively on finite-dimensional normed spaces, the \names also make sense for infinite-dimensional Banach spaces. Accordingly, given Banach spaces $X$, $Y$ and $T :X \to Y$ a bounded linear operator, we may define $\tau_k(T)$ for $k \in \N \cup \{+\iy\}$ exactly as in the finite-dimensional case (note that in that, case, the supremum in \eqref{eq:def-tau} is taken over $z$ in the algebraic product $X^{\otimes k}$). The elementary inequality $\tau_\infty(T) \leq \|T\|_N$ holds in generality (the nuclear norm $\|\cdot\|_N$ is defined in \cite[p.41]{Ryan02}), and nuclear operators have a finite \name. The class of operators for which $\tau_k$ is finite (at fixed $k$) has been discussed in a series of papers by John \cite{KJohn1,KJohn2,KJohn3}. 

Whenever $X$ is an infinite-dimensional Banach space, we have $\rho_3(X) = +\iy$ (i.e., the injective and projective norms are not equivalent on $X \otimes X \otimes X$) and therefore $\rho_{\infty}(X)=+\iy$. An argument for this is given in \cite[4.5]{KJohn4}. This has to be compared with the famous example by Pisier \cite{Pisier83} of an infinite-dimensional space $X$ for which $\rho_2(X) < +\iy$, answering a question by Grothendieck \cite{Grothendieck53}.

We now explain how Theorem \ref{thm:tauRedIdentity} translates to the infinite-dimensional setting. Consider two Banach spaces $X$ and $Y$ and assume that one of them is a Hilbert space. Let $T : X \to Y$ be a bounded operator. An alternative description of $\|T\|_{N}$ can be given by trace duality: it is equal to the \emph{integral norm} of $T$, defined as 
\begin{equation} \label{def:integral-norm} \|T\|_{I} = \sup \{ \Tr \lbr QT\rbr : \|Q\|_{Y \to X} \leq 1, \ \mathrm{rank}(Q) < \iy \}. \end{equation}
We point out that, for general Banach spaces, the integral and nuclear norm are not equal. However, both quantities coincide when $X$ or $Y$ is a Hilbert space (or, more generally, when $X^*$ or $Y$ has the metric approximation property, see \cite[Corollary 4.17]{Ryan02}).

\begin{thm} \label{thm:infinite-dim}
Let $H$ be a Hilbert space, $X$ be a Banach space. Let $T \in L(X,H)$ and $S \in L(H,X)$ be bounded operators. Then $\tau_\infty(T) = \|T\|_N$ and $\tau_\infty(S) = \|S\|_N$.
\end{thm}

\begin{proof}
For every finite rank operators $Q:H\to X$ and $R : X \to H$ such that $\|Q\| \leq 1$, $\|R\| \leq 1$, we have
\[ \Tr[TQ] \leq \tau_\infty(T) \textnormal{ and } \Tr[RS] \leq \tau_\infty(S). \]
These inequalities are obtained by mimicking the proof of Theorem \ref{thm:tauRedIdentity}, identifying $TQ$ or $RS$ as operator on $\ell_2^n$ for some $n$. The result follows from \eqref{def:integral-norm}.
\end{proof}

\section*{Acknowledgements}
We thank Gilles Pisier for useful comments and for pointing out articles by Kamil John. This work was supported in part by ANR (France) under the grant ESQuisses (ANR-20-CE47-0014-01). We acknowledge funding from the European Union’s Horizon 2020 research and innovation programme under the Marie Sk\l odowska-Curie Action TIPTOP (grant no. 843414).

\bibliographystyle{alpha}
\bibliography{Biblio.bib}

\end{document}